\documentclass[a5paper, 10pt]{article}

\usepackage{tikz,amsmath,amssymb,amsfonts,amscd,cite,booktabs,relsize}
\usepackage{lastpage,graphicx,eucal,adjustbox}
\usepackage[hidelinks]{hyperref}
\usepackage{subfigure}
\usepackage{geometry}
\usepackage{lipsum}
\usepackage{setspace}

\usepackage{amsthm}

\theoremstyle{plain}
\newtheorem{lemma}{Lemma}
\newtheorem{theorem}{Theorem}

\theoremstyle{remark}

\theoremstyle{definition}

\newtheorem{claim}{Claim}

\usetikzlibrary{decorations.pathreplacing}
\allowdisplaybreaks

\usepackage{graphicx}
\makeatletter
\renewcommand{\maketitle}{
	\begin{center}

		\baselineskip=0.30in
		{\Large\bfseries \@title} \par
		\vspace{5mm}
		\baselineskip=0.2in
		{\large\bfseries \@author}\par
		\vspace{1mm}
		{\it \@address} \par
		{\small\tt \@email} \par
		\vspace{3mm}
		{\small } \par
	\end{center}
	\vspace{3mm}
}
\newcommand{\address}[1]{\def\@address{#1}}
\newcommand{\email}[1]{\def\@email{#1}}

\address{}

\makeatother

\usepackage{fancyhdr}

\usepackage[margin=1cm,%
			font=footnotesize,%
			format=hang,%
			labelsep=period,%
			labelfont=bf]{caption}

\title{On the Augmented Sombor Index of Graphs}
\author{Kinkar Chandra Das$^{a}$, Akbar Ali$^{b,}$\footnote{Corresponding author.}}

\address{$^a$Department of Mathematics,\\ Sungkyunkwan University, Suwon 16419, Republic of Korea\\
	$^b$Department of Mathematics, College of Science,\\ University of Ha\!'il, Ha\!'il, Saudi Arabia\\
	}

\email{kinkardas2003@gmail.com, akbarali.maths@gmail.com}

\begin{document}

\maketitle

\begin{abstract}
Let $G$ be a connected graph having more than two vertices and let $d_i$ denote the degree of vertex $v_i$ in $G$. Let $E(G)$ represent the edge set of $G$. Then, the augmented Sombor (ASO) index of $G$ is defined as
$ASO(G) = \sum_{v_i v_j \in E(G)} \sqrt{(d_i + d_j - 2)^{-1}(d_i^2 + d_j^2)}.$ It is known that the cycle graph $C_n$ uniquely minimizes the ASO index in the class of all $n$-order unicyclic graphs. In this paper, we prove that the unique $n$-order unicyclic graph of maximum degree $n-1$ maximizes the ASO index in the aforementioned unicyclic graph class.
We also prove that $ASO(G-v_iv_j)<ASO(G)$
whenever neither of the graphs $G-v_iv_j$ and $G$ contains any isolated edge. Utilizing this edge-deletion property, we characterize the unique graph maximizing the ASO index among all fixed-order connected graphs with a specified vertex connectivity (or edge connectivity).
\end{abstract}

\onehalfspacing

\section{Introduction}

Chemical graph theory constitutes a significant subfield of mathematical chemistry in which graph-theoretic methods are employed to represent, analyze, and predict the physical and chemical characteristics of chemical compounds \cite{Leite-24,Wagner-18,Trina-book}. In a chemical graph, vertices model atoms while edges correspond to chemical bonds, thereby providing a mathematically rigorous framework for the study of molecular structure and related properties. Within this setting, real-valued graph invariants are commonly referred to as topological indices.

Among several families of topological indices investigated to date, degree-based indices \cite{KD1,gutman13degree,Ali-CJC-16,borovicanin17zagreb,Ali-IJACM-17} play a particularly prominent role. These indices depend exclusively on the degrees of the vertices, thereby capturing essential local connectivity features of the underlying molecular graph. Their demonstrated success in correlating with a variety of physicochemical properties, together with their computational efficiency, has established them as indispensable tools in quantitative structure-property relationship (QSPR) studies, an area that underpins the prediction and rational design of new chemical compounds \cite{Desmecht-JCIM-24}.

A particularly influential member of the family of degree-based topological indices is the Sombor index, introduced by Gutman \cite{gutman21geo} around half a decade ago. For a graph $G$, this index is defined as
\[
SO(G)=\sum_{v_i v_j \in E(G)} \sqrt{d_i^{\,2} + d_j^{\,2}},
\]
where $E(G)$ denotes the edge set of $G$, and $d_i$ and $d_j$ represent the degrees of the adjacent vertices $v_i$ and $v_j$, respectively. Since its introduction, the Sombor index has spurred numerous studies dedicated to its chemical applications and mathematical aspects; see, for instance, \cite{nithyaa24,CAM3,chen24,CAM4,cruz21unibicy,cruz21trees,das21symmetry,das22trees,das21extremal,deng21moleculartrees,gutman21geo,gutman24eu,liu22review,liu22tetracyclic,cruz21chemical,Milovanovic-21,horoldagvaa21sombor,das24open}.
Motivated by these developments, a new variant of the Sombor index, termed the Augmented Sombor (ASO) index, was proposed in \cite{DGA1} by considering not only the vertex degrees $d_i$ and $d_j$ of the endvertices of each edge $v_iv_j$, but also its edge degree $d_i+d_j-2$. For any graph $G$ that contains no connected component isomorphic to the path graph $P_2$ of order $2$, the ASO index is defined by
\[
ASO(G)
= \sum_{v_i v_j \in E(G)}
  \sqrt{\frac{d_i^{\,2} + d_j^{\,2}}{\,d_i + d_j - 2\,}}.
\]

We now define the terms and notation that are used in the remainder of this paper. The
graph-theoretical terminology used in this paper, but not explicitly defined here, can be found in some standard books such as \cite{Bondy-book,Chartrand-16}. For a vertex $v_i \in V(G)$, let $
N_G(v_i)$ denote the set consisting of all vertices adjacent to $v_i$.
Throughout this paper, we adhere to the conventional notation of graph theory: for example, $C_n$ and $K_n$ denote the cycle and complete graph of order $n$, respectively. A vertex of degree one is called a pendant vertex, and an edge incident with such a vertex is called a pendant edge. An edge whose endvertices are both pendant vertices is referred to as an isolated edge. If $v_i v_j$ is an edge of a graph $G$, then $G - v_i v_j$ denotes the graph obtained from $G$ by deleting the edge $v_i v_j$.
For two vertex-disjoint graphs $G$ and $H$, their disjoint union is written as $G \cup H$. The join of $G$ and $H$, denoted $G \vee H$, is the graph formed by taking $G \cup H$ and adding all possible edges between every vertex of $G$ and every vertex of $H$.
The vertex connectivity of a nontrivial connected graph $G$ is defined as the minimum number of vertices whose removal results in a graph that is either disconnected or trivial. Analogously, the edge connectivity of $G$ is the minimum number of edges whose deletion disconnects the graph. A connected graph whose order and size coincide is referred to as a unicyclic graph.

It is known \cite{DGA1} that the cycle graph $C_n$ uniquely minimizes the ASO index in the class of all $n$-order unicyclic graphs. In this paper, we prove that the unique $n$-order unicyclic graph of maximum degree $n-1$ maximizes the ASO index in the aforementioned unicyclic graph class.
We also prove that $ASO(G-v_iv_j)<ASO(G)$
whenever neither of the graphs $G-v_iv_j$ and $G$ contains any isolated edge. Utilizing this edge-deletion property, we show that $(K_1\cup K_{n-k-1})\vee K_k$ is the unique graph maximizing the ASO index among all $n$-order connected graphs with vertex connectivity (or edge connectivity) $k$, where $n\ge4$.

\section{Unicyclic Graphs}

In this section, we prove that the unique $n$-order unicyclic graph of maximum degree $n-1$ maximizes the ASO index among all $n$-order unicyclic graphs. To achieve this, we first establish the following preliminary result.

\begin{lemma} \label{a1} For $1\leq x\leq a$ and $a>1$,
$$\frac{x^2+a^2}{x+a-2}\leq a+1+\frac{2}{a-1}$$
with equality if and only if $x=1$.
\end{lemma}

\begin{proof} Let
  $$f(x)=\frac{x^2+a^2}{x+a-2},~~~1\leq x\leq a,~a>1.$$
We have
  $$f'(x)=\frac{(x+a-2)\,2x-(x^2+a^2)}{(x+a-2)^2}=\frac{x^2+2(a-2)\,x-a^2}{(x+a-2)^2}.$$
Then one can easily see that $f(x)$ is a increasing function on $x\geq -a+2+\sqrt{2(a^2-2a+2)}$ and a decreasing function on $1\leq x\leq -a+2+\sqrt{2(a^2-2a+2)}$. Since $1\leq x\leq a$, we obtain
  $$f(x)\leq \max\{f(1),\,f(a)\}=f(1)=a+1+\frac{2}{a-1}$$
with equality if and only if $x=1$.
\end{proof}
We define a function $h$ as  $h(d_i, d_j) = \frac{d_i^2 + d_j^2}{d_i + d_j - 2}$.
\begin{lemma}{\rm \cite{DGA1}} \label{ee1} Let $G$ be a graph of order $n\,(>8)$ with any edge $v_iv_j$. Then
$h(d_i,d_j)<h(n-2,n-2)<h(n-1,n-3)=h(n-2,1)<h(n-1,2)<h(n-1,n-2)<h(n-1,n-1)<h(n-1,1)$
for $(d_i,d_j)\notin \Big\{(n-1,1),\,(n-1,2),\,(n-2,1),\,(n-2,n-2),\,(n-1,n-3),\,(n-1,n-2),\,(n-1,n-1)\Big\}$.
\end{lemma}

\begin{theorem} Let $G$ be a unicyclic graph of order $n$. Then
\begin{align}
ASO(G)\leq (n-3)\,\sqrt{n+\frac{2}{n-2}}+2\,\sqrt{n-1+\frac{4}{n-1}}+2\label{kin1}
\end{align}
with equality if and only if $G\cong S'_n$.
\end{theorem}

\begin{proof} Let $v_s$ be the maximum degree vertex of degree $\Delta$ in $G$. Then $d_s=\Delta$. If $\Delta=n-1$, then $G\cong S'_n$ with
$$ASO(G)=(n-3)\,\sqrt{n+\frac{2}{n-2}}+2\,\sqrt{n-1+\frac{4}{n-1}}+2$$
and hence the equality in \eqref{kin1} holds. Otherwise, $\Delta\leq n-2$. For $n\leq 9$, by Sage \cite{SA}, one can easily check that the result (\ref{kin1}) strictly holds. So we now assume that $n\geq 10$. By Lemma \ref{ee1}, for any non-pendant edge $v_iv_j\in E(G)$ satisfying $2\leq d_j\leq d_i\leq n-2$, we obtain
$$\sqrt{\frac{d^2_i+d^2_j}{d_i+d_j-2}}\le \sqrt{h(n-2,n-2)}=\sqrt{n-1+\frac{1}{n-3}}<\sqrt{n-1+\frac{2}{n-3}}.$$
For any pendant edge $v_iv_j\in E(T)$ satisfying $1=d_j<d_i\le n-2$,
$$\sqrt{\frac{d^2_i+d^2_j}{d_i+d_j-2}}<\sqrt{h(n-2,1)}=\sqrt{n-1+\frac{2}{n-3}}.$$
From the above two inequalities, for any edge $v_iv_j\in E(G)$ satisfying $1\leq d_j\leq d_i\leq n-2$, we obtain
\begin{align}
\sqrt{\frac{d^2_i+d^2_j}{d_i+d_j-2}}<\sqrt{n-1+\frac{2}{n-3}}.\label{1r0}
\end{align}
 Let $C_k\,(\geq 3)$ be the cycle in $G$. Also, let $v_t$ be the second maximum degree vertex of degree $\Delta_2$ in $G$. Then $d_t=\Delta_2$ and $2\leq \Delta_2\leq \frac{n-k+4}{2}$. By Lemma \ref{a1}, for $v_i\in N_G(v_t)$ with $d_i\leq \Delta_2$, we obtain
\begin{align}
\frac{\Delta^2_2+d^2_i}{\Delta_2+d_i-2}\leq \Delta_2+1+\frac{2}{\Delta_2-1}\label{1r1}
\end{align}
with equality if and only if $d_i=1$. Let $S=\{v_iv_t\in E(G):\,v_i\in N_G(v_t)\}$. We consider the following three cases:

\vspace*{3mm}

\noindent
${\bf Case\,1.}$ $\Delta_2=2$. In this case, the maximum degree vertex $v_s$ lies on the cycle $C_k$, and all other vertices on $C_k$ have degree $2$. Hence, there exists an edge $v_pv_q\in E(C_k)$ such that $d_p=d_q=2$. For this edge, we obtain
    $$\sqrt{\frac{d^2_p+d^2_q}{d_p+d_q-2}}=2.$$
Using this fact together with (\ref{1r0}), we obtain
\begin{align*}
ASO(G)
&=\sqrt{\frac{d^2_p+d^2_q}{d_p+d_q-2}}+\sum\limits_{v_iv_j\in E(G)\backslash \{v_pv_q\}}\,\sqrt{\frac{d^2_i+d^2_j}{d_i+d_j-2}}\\[3mm]
&<2+(n-1)\sqrt{n-1+\frac{2}{n-3}}\\[3mm]
&<(n-3)\,\sqrt{n+\frac{2}{n-2}}+2\,\sqrt{n-1+\frac{4}{n-1}}+2.
\end{align*}
Thus, the inequality (\ref{kin1}) holds strictly in the present case.

\vspace*{3mm}

In the rest of the cases, the second maximum degree vertex $v_t$ is adjacent to at least $(\Delta_2-1)$ vertices of degree at most $\Delta_2$.

\vspace*{3mm}
\noindent
${\bf Case\,2.}$ $3\leq \Delta_2\leq 5$.  Then, by (\ref{1r0}) and (\ref{1r1}), we obtain
    $$\sum\limits_{v_i\in N_G(v_t)}\,\sqrt{\frac{d^2_i+\Delta^2_2}{d_i+\Delta_2-2}}<\sqrt{n-1+\frac{2}{n-3}}+\sqrt{\frac{13}{2}}\,(\Delta_2-1).$$
Using this with (\ref{1r0}), we obtain
\begin{align}
ASO(G)&=\sum\limits_{v_iv_j\in E(G)}\,\sqrt{\frac{d^2_i+d^2_j}{d_i+d_j-2}}\nonumber\\[3mm]
  &=\sum\limits_{v_i\in N_G(v_t)}\,\sqrt{\frac{d^2_i+\Delta^2_2}{d_i+\Delta_2-2}}+\sum\limits_{v_iv_j\in E(G)\backslash S}\,\sqrt{\frac{d^2_i+d^2_j}{d_i+d_j-2}}\nonumber\\[3mm]
  &<\sqrt{n-1+\frac{2}{n-3}}+\sqrt{\frac{13}{2}}\,(\Delta_2-1)+(n-\Delta_2)\,\sqrt{n-1+\frac{2}{n-3}}.\label{1vv1}
\end{align}
We now prove the following claim:

\vspace*{3mm}

\begin{claim}\label{c1}
\begin{align*}
&\sqrt{\frac{13}{2}}\,(\Delta_2-1)+(n-\Delta_2)\,\sqrt{n-1+\frac{2}{n-3}}\\&<(n-3)\,\sqrt{n+\frac{2}{n-2}}+\sqrt{n-1+\frac{2}{n-3}}+2.
\end{align*}

\end{claim}

\noindent
{\bf Proof of Claim \ref{c1}.} Since $n\geq 10$, for $4\leq \Delta_2 \leq 5$, we obtain
\begin{align*}
&\sqrt{\frac{13}{2}}\,(\Delta_2-1)+(n-\Delta_2)\,\sqrt{n-1+\frac{2}{n-3}}\\&\leq 3\,\sqrt{\frac{13}{2}}+(n-4)\,\sqrt{n-1+\frac{2}{n-3}}\\[3mm]
&<(n-3)\,\sqrt{n+\frac{2}{n-2}}+\sqrt{n-1+\frac{2}{n-3}}+2
\end{align*}
as
  $$3\,\sqrt{\frac{13}{2}}<8<2\,\sqrt{n-1+\frac{2}{n-3}}+2.$$
Otherwise, $\Delta_2=3$. Now,
\begin{align*}
&\sqrt{\frac{13}{2}}\,(\Delta_2-1)+(n-\Delta_2)\,\sqrt{n-1+\frac{2}{n-3}}\\&=2\,\sqrt{\frac{13}{2}}+(n-3)\,\sqrt{n-1+\frac{2}{n-3}}\\[3mm]
&<(n-3)\,\sqrt{n+\frac{2}{n-2}}+\sqrt{n-1+\frac{2}{n-3}}+2
\end{align*}
as
$n\geq 10,~2\,\sqrt{\frac{13}{2}}<\frac{516}{100}<\sqrt{n+\frac{2}{n-2}}+2$ and
$$\sqrt{n-1+\frac{2}{n-3}}<\sqrt{n+\frac{2}{n-2}}.$$
Hence, {\bf Claim \ref{c1}} holds.

\vspace*{3mm}

Using {\bf Claim \ref{c1}} in (\ref{1vv1}), we obtain
  \begin{align*}
ASO(G)&<(n-3)\,\sqrt{n+\frac{2}{n-2}}+2\,\sqrt{n-1+\frac{2}{n-3}}+2\\&<(n-3)\,\sqrt{n+\frac{2}{n-2}}+2\,\sqrt{n-1+\frac{4}{n-1}}+2.  \end{align*}
Thus, in the present case, (\ref{kin1}) holds strictly.

\vspace*{3mm}

\noindent
${\bf Case\,3.}$ $\Delta_2=6$. Then, by (\ref{1r0}) and (\ref{1r1}), we obtain
   \begin{align*}
       \sum\limits_{v_i\in N_G(v_t)}\,\sqrt{\frac{d^2_i+\Delta^2_2}{d_i+\Delta_2-2}}&<\sqrt{n-1+\frac{2}{n-3}}+\sqrt{\frac{37}{5}}\,(\Delta_2-1)\\&=\sqrt{n-1+\frac{2}{n-3}}+5\,\sqrt{\frac{37}{5}}.   \end{align*}
Since $n\geq 10$, using the above inequality with (\ref{1r0}), we obtain
\begin{align*}
ASO(G)&=\sum\limits_{v_i\in N_G(v_t)}\,\sqrt{\frac{d^2_i+\Delta^2_2}{d_i+\Delta_2-2}}+\sum\limits_{v_iv_j\in E(G)\backslash S}\,\sqrt{\frac{d^2_i+d^2_j}{d_i+d_j-2}}\\[3mm]
  &<\sqrt{n-1+\frac{2}{n-3}}+5\,\sqrt{\frac{37}{5}}+(n-6)\,\sqrt{n-1+\frac{2}{n-3}}\\[3mm]
  &<\frac{1361}{100}+(n-5)\,\sqrt{n-1+\frac{2}{n-3}}\\[3mm]
  &<(n-1)\,\sqrt{n-1+\frac{2}{n-3}}+2\\[3mm]
  &<(n-3)\,\sqrt{n+\frac{2}{n-2}}+2\,\sqrt{n-1+\frac{4}{n-3}}+2
\end{align*}
as $n\geq 10$ and
$$4\,\sqrt{n-1+\frac{2}{n-3}}+2>14>\frac{1361}{100}.$$
The inequality (\ref{kin1}) holds strictly.

\vspace*{3mm}

\noindent
${\bf Case\,4.}$ $7\leq \Delta_2\leq \frac{n-k+4}{2}$. Now,
\begin{align}
ASO(G)&=\sum\limits_{v_i\in N_G(v_t)}\,\sqrt{\frac{d^2_i+\Delta^2_2}{d_i+\Delta_2-2}}+\sum\limits_{v_iv_j\in E(G)\backslash S}\,\sqrt{\frac{d^2_i+d^2_j}{d_i+d_j-2}}.\label{1r2}
\end{align}

\vspace*{3mm}

\noindent
${\bf Case\,4.1.}$ $7\leq \Delta_2\leq \frac{n-3}{2}$. In this case $n\geq 17$. We note that $\Delta_2+1+\frac{2}{\Delta_2-1}$ attains its maximum value over $7\leq \Delta_2\leq \frac{n-3}{2}$ at $\Delta_2= \frac{n-3}{2}$. Hence, for $d_i\leq \Delta_2$, from \eqref{1r1}, we obtain
\begin{equation}\label{eq-Ak-1}
     \frac{\Delta^2_2+d^2_i}{\Delta_2+d_i-2}\leq \frac{n-1}{2}+\frac{4}{n-5}<\frac{1}{2}\,\Big(n+\frac{2}{n-2}\Big)
     \end{equation}
as $n\geq 17$. Recall that the second maximum degree vertex $v_t$ is adjacent to at least $(\Delta_2-1)$ vertices of degree at most $\Delta_2$. Hence,  using (\ref{eq-Ak-1}) and (\ref{1r0}), we obtain
\begin{equation}\label{eq-ak-2}
\sum\limits_{v_i\in N_G(v_t)}\,\sqrt{\frac{d^2_i+\Delta^2_2}{d_i+\Delta_2-2}}<(\Delta_2-1)\,\sqrt{\frac{1}{2}\,\Big(n+\frac{2}{n-2}\Big)}+\sqrt{n-1+\frac{2}{n-3}}.
\end{equation}
Since $\Delta_2\geq 7$, using (\ref{1r0}) and (\ref{eq-ak-2}) in (\ref{1r2}), we obtain
\begin{align*}
ASO(G)&<(\Delta_2-1)\,\sqrt{\frac{1}{2}\,\Big(n+\frac{2}{n-2}\Big)}+\sqrt{n-1+\frac{2}{n-3}}\\[3mm]
&\quad+(n-\Delta_2)\,\sqrt{n-1+\frac{2}{n-3}}\\[3mm]
 &\leq 6\,\sqrt{\frac{1}{2}\,\Big(n+\frac{2}{n-2}\Big)}+(n-6)\,\sqrt{n-1+\frac{2}{n-3}}\\[3mm]
 &<5\,\sqrt{n+\frac{2}{n-2}}+(n-6)\,\sqrt{n-1+\frac{2}{n-3}}\\[3mm]
 &<(n-3)\,\sqrt{n+\frac{2}{n-2}}+2\,\sqrt{n-1+\frac{4}{n-3}}+2
\end{align*}
as $n\geq 17$ and
$\sqrt{n+\frac{2}{n-2}}>\sqrt{n-1+\frac{2}{n-3}}<\sqrt{n-1+\frac{4}{n-3}}.$
Hence, the inequality (\ref{kin1}) strictly holds.

\vspace*{3mm}

\noindent
${\bf Case\,4.2.}$ $\Delta_2>\frac{n-3}{2}$. We have $\frac{n-3}{2}<\Delta_2\leq \frac{n-k+4}{2}\leq \frac{n+1}{2}$ as $k\geq 3$. Since $G$ is unicyclic, we have $\Delta<\frac{n+5}{2}$. We note that $\Delta_2+1+\frac{2}{\Delta_2-1}$ attains its maximum value over $\frac{n-3}{2}<\Delta_2\leq \frac{n+1}{2}$ at $\Delta_2= \frac{n+1}{2}$. Hence, for $d_i\le \Delta_2$, from (\ref{1r1}), we obtain
     $$\frac{\Delta^2_2+d^2_i}{\Delta_2+d_i-2}\leq \frac{n+3}{2}+\frac{4}{n-1}<\frac{1}{1.44}\,\Big(n+\frac{2}{n-2}\Big)$$
as $n\geq 10$. Using this result with (\ref{1r0}), we obtain
$$\sum\limits_{v_i\in N_G(v_t)}\,\sqrt{\frac{d^2_i+\Delta^2_2}{d_i+\Delta_2-2}}<\frac{\Delta_2-1}{1.2}\,\sqrt{n+\frac{2}{n-2}}+\sqrt{n-1+\frac{2}{n-3}}.$$
Since $\Delta<\frac{n+5}{2}$, by Lemma \ref{a1}, for $v_i\in N_G(v_s)$, we obtain
\begin{align}
\frac{\Delta^2+d^2_i}{\Delta+d_i-2}&\leq \Delta+1+\frac{2}{\Delta-1}<\frac{n+7}{2}+\frac{4}{n+3}<n+\frac{2}{n-2}\label{1r4}
\end{align}
as $n\geq 10$. Since $\Delta_2\geq 7$, using the above results with (\ref{1r0}), from (\ref{1r2}), we obtain
\begin{align*}
ASO(G)&<\frac{(\Delta_2-1)}{1.2}\,\sqrt{n+\frac{2}{n-2}}+\sqrt{n-1+\frac{2}{n-3}}\\[3mm]
 &\quad+(n-\Delta_2)\,\sqrt{n-1+\frac{2}{n-3}}\\[3mm]
 &\leq \frac{6}{1.2}\,\sqrt{n+\frac{2}{n-2}}+(n-6)\,\sqrt{n-1+\frac{2}{n-3}}\\[3mm]
 &=5\,\sqrt{n+\frac{2}{n-2}}+(n-6)\,\sqrt{n-1+\frac{2}{n-3}}\\[3mm]
 &<(n-3)\,\sqrt{n+\frac{2}{n-2}}+2\,\sqrt{n-1+\frac{4}{n-3}}+2.
\end{align*}
Therefore, in the present case, the inequality (\ref{kin1}) strictly holds.
\end{proof}

\section{Edge-Deletion Property}\label{sec-3}

In this section, we prove that $ASO(G-v_iv_j)<ASO(G)$ whenever neither of the graphs $G-v_iv_j$ and $G$ contains any isolated edge. To establish this result, we first present the following lemma:
\begin{lemma}\label{lem-1}
The inequality
\begin{equation}\label{eq-lem-1}
\frac{x^{2}+2xy-y^{2}-4x}{(x+y-2)^{3/2}\sqrt{x^2+y^2}}
> -\frac{1}{\sqrt{x}}
\end{equation}
holds for all real numbers $x$ and $y$ satisfying $x\ge2$ and $y\ge1$.
\end{lemma}

\begin{proof}
    Consider
    \[
f(x,y)=x^{2}+2xy-y^{2}-4x +\frac{1}{\sqrt{x}} (x+y-2)^{3/2}\sqrt{x^{2}+y^{2}}
,   \]
where $x\ge2$ and $y\ge1$.
Then, we have to prove that $f(x,y)> 0$. If $1\le y \le x$ and $x\ge2$, then we have
\begin{align*}
\frac{\partial f}{\partial y}&=    \frac{y (x+y-2)^{3/2}}{ \sqrt{x(x^2+y^2)}}+\frac{3 \sqrt{x^2+y^2} \sqrt{x+y-2}}{2 \sqrt{x}}+2x-2y > 0.
\end{align*}
If $x< y < 3x$ and $x\ge2$, then
\begin{align*}
\frac{\partial f}{\partial y}&>    \frac{\sqrt{x^2+y^2} \sqrt{x+y-2}}{ \sqrt{x}}+2x-2y\\
&>\frac{\sqrt{2xy} \sqrt{x+y-2}}{ \sqrt{x}}-2(y-x)\\
&=\frac{2y(x+y-2)-4(y-x)^2}{\sqrt{2y(x+y-2)}+2(y-x)}\\
&>\frac{2\left(\frac{3(y-x)}{2}\right)\left(\frac{3(y-x)}{2}\right)-4(y-x)^2}{\sqrt{2y(x+y-2)}+2(y-x)}>0.
\end{align*}
Therefore, if $1\le y < 3x$ and $x\ge2$, then $\frac{\partial f}{\partial y}>0$, and hence
\begin{align*}
f(x,y)&\ge f(x,1)=x^2+\frac{(x-1)^{3/2} \sqrt{x^2+1}}{\sqrt{x}}-2 x-1 >  x^2-2x\ge0.
\end{align*}
It remains to prove $f(x,y)>0$ for $y\ge3x\ge6$.
Since each of the inequalities $x^{2}-4x>-4x^2$, $x+y-2\ge y$ and $\sqrt{x^{2}+y^{2}}> y$ holds for $x\ge2$ and $y\ge1$, the following chain of inequalities holds for $y\ge3x\ge6$:
\begin{align*}
f(x,y)&>2xy-y^2-4x^2 + \frac{y^{5/2}}{\sqrt{x}}\\
&>-y^2-3x^2 + \frac{y^{5/2}}{\sqrt{x}}\\
&\ge -y^2-3\left(\frac{y}{3}\right)^2 + \frac{y^{5/2}}{\sqrt{\frac{y}{3}}}>0.
\end{align*}

\end{proof}

The following result shows that the ASO index decreases (increases, respectively) whenever an edge is removed (added, respectively) to a graph under a mild condition.

\begin{theorem} \label{k1}
Let $G$ be a graph containing no isolated edge. If $v_iv_j\in E(G)$ such that $G-v_iv_j$ contains no isolated edge, then
\[
ASO(G) > ASO(G-v_iv_j).
\]
\end{theorem}

\begin{proof} In this proof, we set \[
f(x,y):=\sqrt{\frac{x^2+y^2}{x+y-2}}.
\]
and we denote by $d_\ell$ the degree of a vertex $v_\ell\in V(G-v_iv_j)=V(G)$ in $G$.
We note that
\begin{align}\label{eq-main_G-uv-}
ASO(G)-ASO(G-v_iv_j)
&=
\sum_{v_r\in N_G(v_i)\setminus\{v_j\}}\big( f(d_i,d_r)-f(d_i-1,d_r)\big)\nonumber\\
&\quad +\sum_{v_s\in N_G(v_j)\setminus\{v_i\}}\big( f(d_j,d_s)-f(d_j-1,d_s)\big)\nonumber\\
&\quad +f(d_i,d_j).
\end{align}
Set
$\Theta_i=\sum_{v_r\in N_G(v_i)\setminus\{v_j\}}\theta_i(v_r)$ and  $\Theta_j=\sum_{v_s\in N_G(v_j)\setminus\{v_i\}}\theta_j(v_s),$
where
$\theta_i(v_r)=f(d_i,d_r)-f(d_i-1,d_r)$ and $\theta_j(v_s)=f(d_j,d_s)-f(d_j-1,d_s).$
Then, \eqref{eq-main_G-uv-} yields
\begin{align}\label{eq-main_G-uv}
ASO(G)-ASO(G-v_iv_j)= \Theta_i + \Theta_j + f(d_i,d_j)
\end{align}
We note that $\max\{d_i,d_j\}\ge2.$
Without loss of generality, we assume that $d_i\ge d_j$. We observe that $\Theta_j=0$ for $d_j=1$. Hence, whenever we consider $\theta_j(v_s),$ it will be assumed that $d_j\ge2$.

By using the mean value theorem, we observe that, for any vertex $v_r\in N_G(v_i)\setminus\{v_j\}$, there exists $\xi_r$ such that $1\le d_i-1 <\xi_r<d_i$ and
\begin{equation}\label{theta_u-w-}
\theta_i(v_r)= \frac{\xi_r^{2}+2\xi_r d_r-d_r^{2}-4\xi_r}{2(\xi_r+d_r-2)^{3/2}\sqrt{\xi_r^2+d_r^2}}.
\end{equation}
Similarly, if $d_j\ge2$, then for any vertex $v_s\in N_G(v_j)\setminus\{v_i\}$, there exists $\xi_s$ such that $1\le d_j-1 <\xi_s<d_j$ and
\begin{equation}\label{theta_v-x-}
\theta_j(v_s)= \frac{\xi_s^{2}+2\xi_s d_s-d_s^{2}-4\xi_s}{2(\xi_s+d_s-2)^{3/2}\sqrt{\xi_s^2+d_s^2}}.
\end{equation}
By Lemma \ref{lem-1}, from \eqref{theta_u-w-} and \eqref{theta_v-x-}, it follows that
\[
\theta_i(v_r)>-\frac{1}{2\sqrt{\xi_r}}> -\frac{1}{2\sqrt{d_i-1}} \
\text{ and } \
\theta_j(v_s)>-\frac{1}{2\sqrt{\xi_s}}> -\frac{1}{2\sqrt{d_j-1}},
\]
provided that $d_j\ge2$.
Therefore,
\[
\Theta_i> -\frac{\sqrt{d_i-1}}{2}
\
\text{ and }
\
\Theta_j\ge-\frac{\sqrt{d_j-1}}{2},
\]
where the last inequality holds not only for $d_j\ge2$ but also for $d_j=1$.
Consequently, we have
\begin{align}\label{last-ineq}
\Theta_i+\Theta_j+f(d_i,d_j)
&> -\frac{\sqrt{d_i-1}+\sqrt{d_j-1}}{2}+\sqrt{\frac{d_i^2+d_j^2}{d_i+d_j-2}}\nonumber\\
&>-\frac{\sqrt{d_i-1}+\sqrt{d_j-1}}{2}+\sqrt{\frac{d_i+d_j}{2}},
\end{align}
where the last inequality in \eqref{last-ineq} holds because $(d_i-d_j)^2+2(d_i+d_j)>0,$ which gives $2[d_i^2+d_j^2]>(d_i+d_j)(d_i+d_j-2).$
Also, we have
\begin{equation}\label{final-ineq-2}
\frac{\sqrt{d_i-1}+\sqrt{d_j-1}}{2}
\le \sqrt{\frac{(d_i-1)+(d_j-1)}{2}}<\sqrt{\frac{d_i+d_j}{2}}.
\end{equation}
Hence, the desired inequality follows from \eqref{eq-main_G-uv}, \eqref{last-ineq} and \eqref{final-ineq-2}.
\end{proof}

\section{Vertex/Edge Connectivity}
In the present section, we show that $(K_1\cup K_{n-k-1})\vee K_k$ is the unique graph maximizing the ASO index among all $n$-order connected graphs with vertex connectivity (or edge connectivity) $k$, where $n\ge4$. 
First, we establish a result concerning the vertex connectivity.

\begin{theorem}\label{k0} Let $G$ be a connected graph of order $n\ge4$ with vertex connectivity $k$. Then
\begin{align}
ASO(G)&\leq k\,\sqrt{\frac{(n-1)^2+k^2}{n+k-3}}+{k\choose 2}\,\sqrt{n+\frac{1}{n-2}}\nonumber\\&\quad+k\,(n-k-1)\,\sqrt{n-\frac{1}{2}+\frac{5}{2(2n-5)}}\nonumber\\
&\quad+{n-k-1\choose 2}\,\sqrt{n-1+\frac{1}{n-3}}\label{1k2}
\end{align}
with equality in (\ref{1k2}) if and only if $G\cong (K_1\cup K_{n-k-1})\vee K_k$ for $k\le n-2$, whereas $G\cong K_n$ for $k=n-1$.
\end{theorem}

\begin{proof} For $k=n-1$, we have $G\cong K_n$ and hence the equality in (\ref{1k2}) holds. In what follows, we assume that $k\le n-2$. If $G\cong (K_1\cup K_{n-k-1})\vee K_k$, then one can easily see that the equality holds in (\ref{1k2}). Otherwise, $G\ncong (K_1\cup K_{n-k-1})\vee K_k$. Since $G$ has vertex connectivity $k$, there exists a
set $S\subset V(G)$ consisting of $k$ vertices such that $G-S$ has at least two components. If $G-A$ consists of at least three components, then inserting an edge connecting the vertices lying in two distinct components of $G-A$ enlarges the value of $ASO(G)$ (by Theorem \ref{k1}) but the vertex connectivity of $G$ remains the same. Thereby, it is sufficient to prove the result when the graph $G-A$ has only two components; we name them as $G_1$ and $G_2$ such that $|V(G_1)|\leq |V(G_2)|$. Suppose that $|V(G_1)|=t$. Then $|V(G_2)|=n-k-t$ and $t\leq \frac{n-k}{2}$. By Theorem \ref{k1}, we obtain
 $$ASO(G)\leq ASO((K_{t}\cup K_{n-k-t})\vee K_k).$$
Since $G\ncong (K_1\cup K_{n-k-1})\vee K_k$, we have $2\leq t\leq \frac{n-k}{2}$, which yields $n\geq 5$ and $k\leq n-4$. Thus, we have $2\leq k+t-1\leq n-t-1\leq n-3$. Hence, by Lemma \ref{ee1}, we obtain
\begin{align*}
&h(k+t-1,k+t-1)<h(n-2,n-2)=n-1+\frac{1}{n-3},\\
&h(n-t-1,n-t-1)<h(n-2,n-2)=n-1+\frac{1}{n-3},\\
&h(n-1,k+t-1)\leq h(n-1,2)=n-1+\frac{4}{n-1},\\
&h(n-1,n-t-1)\leq h(n-1,2)=n-1+\frac{4}{n-1}.
\end{align*}
 Now,
\begin{align}
&ASO((K_{t}\cup K_{n-k-t})\vee K_k)\nonumber\\
&={k\choose 2}\,\sqrt{h(n-1,n-1)}+{n-k-t\choose 2}\,\sqrt{h(n-t-1,n-t-1)}\nonumber\\
&~~~~~~~~+{t\choose 2}\,\sqrt{h(k+t-1,k+t-1)}+kt\,\sqrt{h(n-1,k+t-1)}\nonumber\\
&~~~~~~~~+k(n-k-t)\,\sqrt{h(n-1,n-t-1)}\nonumber\\
&<{k\choose 2}\,\sqrt{n+\frac{1}{n-2}}+{n-k-t\choose 2}\,\sqrt{n-1+\frac{1}{n-3}}\nonumber\\
&~~~~~~~~+{t\choose 2}\,\sqrt{n-1+\frac{1}{n-3}}+kt\,\sqrt{n-1+\frac{4}{n-1}}\nonumber\\
&~~~~~~~~+k(n-k-t)\,\sqrt{n-1+\frac{4}{n-1}}\nonumber\\
&=\left({t\choose 2}+{n-k-t\choose 2}\right)\,\sqrt{n-1+\frac{1}{n-3}}+{k\choose 2}\,\sqrt{n+\frac{1}{n-2}}\nonumber\\
&~~~~~~~~+k(n-k)\,\sqrt{n-1+\frac{4}{n-1}}.\label{p1}
\end{align}
Let us consider a function
  $$f(x)=x\,(x-1)+(n-k-x)\,(n-k-x-1),~~~2\leq x\leq \frac{n-k}{2}.$$
Then we obtain
   $$f'(x)=2\,\Big(2x-(n-k)\Big)\leq 0~\mbox{ as }2\leq x\leq \frac{n-k}{2}.$$
Thus $f(x)$ is a decreasing function on $2\leq x\leq \frac{n-k}{2}$, and hence
\begin{align}
{t\choose 2}+{n-k-t\choose 2}=\frac{1}{2}\,f(t)&\leq \frac{1}{2}\,f(2)=\frac{1}{2}\,\Big((n-k-2)\,(n-k-3)+2\Big)\nonumber\\[3mm]
  &={n-k-1\choose 2}-(n-k-3).\label{p2}
\end{align}
Since
   $$   n-1+\frac{4}{n-1}\leq n-\frac{1}{2}+\frac{5}{2(2n-5)},$$
from (\ref{p1}) and (\ref{p2}), we obtain
\begin{align}\label{eq-Ak-04}
&ASO((K_{t}\cup K_{n-k-t})\vee K_k)\nonumber\\[3mm]
&<{n-k-1\choose 2}\,\sqrt{n-1+\frac{1}{n-3}}+{k\choose 2}\,\sqrt{n+\frac{1}{n-2}}\nonumber\\[3mm]
&~~~~~~~~+k(n-k)\,\sqrt{n-1+\frac{4}{n-1}}-(n-k-3)\,\sqrt{n-1+\frac{1}{n-3}}\nonumber\\[3mm]
&<{n-k-1\choose 2}\,\sqrt{n-1+\frac{1}{n-3}}+{k\choose 2}\,\sqrt{n+\frac{1}{n-2}}\nonumber\\[3mm]
&~~~~~~~~~~
+k\,(n-k-1)\,\sqrt{n-\frac{1}{2}+\frac{5}{2(2n-5)}}\nonumber\\[3mm]
&~~~~~~~~~~+k\,\sqrt{\frac{(n-1)^2+k^2}{n+k-3}}.
\end{align}
For the last inequality in \eqref{eq-Ak-04}, it remains to prove that
\begin{align*}
k\,\sqrt{n-1+\frac{4}{n-1}}-(n-k-3)\,\sqrt{n-1+\frac{1}{n-3}}<k\,\sqrt{\frac{(n-1)^2+k^2}{n+k-3}},
\end{align*}
which is equivalent to
\begin{align}\label{Eq-connectivity-01}
k\left(\sqrt{h(k,n-1)}-\sqrt{h(2,n-1)}\,\right)+(n-k-3)\,\sqrt{h(n-2,n-2)}>0.
\end{align}
If $k=1$ or $k=2$, then by Lemma \ref{ee1},  $\sqrt{h(k,n-1)}-\sqrt{h(2,n-1)}\ge0$, and hence \eqref{Eq-connectivity-01} holds. In what follows, we assume that $k\ge3$. Then, by Lemma \ref{ee1}, we have $\sqrt{h(k,n-1)}-\sqrt{h(2,n-1)}<0$. Hence,
in order to prove \eqref{Eq-connectivity-01}, it is sufficient to show that
\[
(n-4)\left(\sqrt{h(k,n-1)}-\sqrt{h(2,n-1)}\,\right)+(n-k-3)\,\sqrt{h(n-2,n-2)}>0
\]
for $3\le k\leq n-4$.
We define a function $\phi$ on the set of real numbers greater than or equal to 3 as
{\small\begin{align*}
\phi(x)=(n-4)\left(\sqrt{h(x,n-1)}-\sqrt{h(2,n-1)}\,\right)+(n-x-3)\,\sqrt{h(n-2,n-2)},
\end{align*}}
where $n\ge7$ is a fixed integer satisfying $3\le x\leq n-4$. Here, we have
$$\frac{d^2}{dx^2}\left(\sqrt{h(x,n-1)}\,\right)=\frac{\Psi(x)}{4 (n+x-3)^{5/2} \left((n-1)^2+x^2\right)^{3/2}},$$
where
{\small$$\Psi(x)=\left[(n-1)^2 (n (7 n-30)+39)-x^4\right]+4x (n-3)[(n-1)^2 - x^2]+6 (n-1)^2 x^2.$$}
Since $(n-1)^2 (n (7 n-30)+39)>n^4>x^4$ and $(n-1)^2>x^2$, we have $\Psi(x)>0$ and hence
$$\frac{d^2}{dx^2}\left(\sqrt{h(x,n-1)}\,\right)>0.$$
Consequently, we obtain
\begin{align*}
\phi'(x)&\le(n-4)\left[\frac{d}{dx}\left(\sqrt{h(x,n-1)}\,\right)\right]_{x=n-4}-\sqrt{h(n-2,n-2)}\\[2mm]
&=\frac{(n-4) (2n^2-20n+39)}{2(2n-7)^{3/2} \sqrt{2n^2-10n+17}}- \sqrt{n-1+\frac{1}{n-3}}\\[2mm]
&<\frac{(n-4) (2n^2-20n+39)}{2(2n-8)^{3/2} \sqrt{2n(n-5)+17}}- \sqrt{n-1+\frac{1}{n-3}}\\[2mm]
&=\frac{\left(2n^2 -20n+39\right) \sqrt{2(n-3)} -8(n-2) \sqrt{n-4} \sqrt{2n(n-5)+17}}{8 \sqrt{(n-4)(n-3)} \sqrt{2n(n-5)+17}}\\[2mm]
&<\frac{\left(2n^2 -20n+39\right) 2\sqrt{n-4} -8(n-2) \sqrt{n-4} \sqrt{2n(n-5)+17}}{8 \sqrt{(n-4)(n-3)} \sqrt{2n(n-5)+17}}\\&<0,
\end{align*}
where the last inequality holds because $2n(n-5)+17>(n-2)^2$ and $2n^2 -20n+39<4(n-2)^2$ for $n\ge7$.
Therefore, by keeping in mind Lemma \ref{ee1}, we have
\begin{align}\label{Eq-con-last}
&\phi(x)\ge \phi(n-4)\nonumber\\&=(n-4)\left(\sqrt{h(n-4,n-1)}-\sqrt{h(2,n-1)}\,\right)+\sqrt{h(n-2,n-2)}\nonumber\\
&>(n-3)\sqrt{h(n-4,n-1)}-(n-4)\sqrt{h(2,n-1)} >0.
\end{align}
The last inequality in \eqref{Eq-con-last} holds because, for $n\ge7$,
\[
\left(\frac{n-3}{n-4}\right)^2>\frac{h(2,n-1)}{h(n-4,n-1)},
\]
which is equivalent to $$-3n^3(n - 9) -n(116n-319)-407<0.$$
Therefore, $\phi(k)>0$, which implies \eqref{Eq-connectivity-01}.
This completes the proof of the theorem.
\end{proof}

We now establish a result concerning the edge connectivity.

\begin{theorem} Let $G$ be a connected graph of order $n\ge4$ with edge connectivity $k'$. Then
\begin{align}
ASO(G)&\leq k'\,\sqrt{\frac{(n-1)^2+k'^2}{n+k'-3}}+{k'\choose 2}\,\sqrt{n+\frac{1}{n-2}}\nonumber\\[3mm]
&\quad+k'\,(n-k'-1)\,\sqrt{n-\frac{1}{2}+\frac{5}{2(2n-5)}}\nonumber\\[3mm]
&\quad+{n-k'-1\choose 2}\,\sqrt{n-1+\frac{1}{n-3}}\label{1k2-e}
\end{align}
with equality in (\ref{1k2-e}) if and only if $G\cong (K_1\cup K_{n-k'-1})\vee K_{k'}$ for $k'\le n-2$, whereas $G\cong K_n$ for $k'=n-1$.
\end{theorem}

\begin{proof} For $k'=n-1$, we have $G\cong K_n$ and hence, the equality in (\ref{1k2-e}) holds. In what follows, we assume that $1\leq k'\leq n-2$. Thus, we have $1\leq k\leq k'\leq n-2$.
By the proof of Lemma \ref{a1}, the function $$f(x)=\displaystyle{\sqrt{\frac{(n-1)^2+x^2}{n+x-3}}}$$ is a strictly increasing function on $x\geq -n+3+\sqrt{2(n^2-4n+5)}$ and a strictly decreasing function on $1\leq x\leq -n+3+\sqrt{2(n^2-4n+5)}$. Thus we have
\begin{align}
\sqrt{\frac{(n-1)^2+x^2}{n+x-3}}&\geq f\Big(-n+3+\sqrt{2(n^2-4n+5)}\Big)\nonumber\\[3mm]
     &=\sqrt{\frac{4\,(n^2-4n+5)-2\,(n-3)\,\sqrt{2(n^2-4n+5)}}{\sqrt{2(n^2-4n+5)}}}\nonumber\\[3mm]
     &=\sqrt{2\sqrt{2(n^2-4n+5)}-2\,(n-3)}\nonumber\\[2mm]
     &>\sqrt{2\,\sqrt{2}\,(n-2)-2n+6}>0.9\,\sqrt{n}.\label{b2}
\end{align}
For $1\leq x\leq n-2$, we obtain
\begin{align}
x\,\Big(x\,(2n+x-6)-(n-1)^2\Big)&>-(n-1)^2\,x>-1.8\,\sqrt{n}\,(n-1)(n-2)^{3/2}\nonumber\\[2mm]
  &>-1.8\,\sqrt{n}\,\sqrt{(n-1)^2+x^2}\,(n+x-3)^{3/2},\nonumber
\end{align}
that is,
\begin{align}
\frac{x\,\Big(x\,(2n+x-6)-(n-1)^2\Big)}{2\,\sqrt{(n-1)^2+x^2}\,(n+x-3)^{3/2}}> -0.9\,\sqrt{n}.\label{b3}
\end{align}
Let us consider a function
\begin{align*}
g(x)&={x\choose 2}\,\sqrt{n+\frac{1}{n-2}}+x\,(n-x-1)\,\sqrt{n-\frac{1}{2}+\frac{5}{2(2n-5)}}\nonumber\\[3mm]
&\quad + x\,\sqrt{\frac{(n-1)^2+x^2}{n+x-3}}+{n-x-1\choose 2}\,\sqrt{n-1+\frac{1}{n-3}},
\end{align*}
for $1\leq x\leq n-2$.
If $1\le x \le \frac{n-1}{2}$, then using (\ref{b2}) and (\ref{b3}),  we have
\begin{align*}
&g'(x)=\sqrt{\frac{(n-1)^2+x^2}{n+x-3}}+\frac{x\,\Big(x\,(2n+x-6)-(n-1)^2\Big)}{2\,\sqrt{(n-1)^2+x^2}\,(n+x-3)^{3/2}}\nonumber\\[3mm]
&\quad+\Big(x-\frac{1}{2}\Big)\,\sqrt{n+\frac{1}{n-2}}+(n-2x-1)\,\sqrt{n-\frac{1}{2}+\frac{5}{2(2n-5)}}\nonumber\\[3mm]
&\quad-\Big((n-x-1)-\frac{1}{2}\Big)\,\sqrt{n-1+\frac{1}{n-3}}\nonumber\\[3mm]
&>\left(\Big(x-\frac{1}{2}\Big)+(n-2x-1)-(n-x-1)+\frac{1}{2}\right)\,\sqrt{n-1+\frac{1}{n-3}}=0.
\end{align*}
If $x > \frac{n-1}{2}$, then using (\ref{b2}) and (\ref{b3}),  we have
\begin{align*}
g'(x)&>\left(\Big(x-\frac{1}{2}\Big)+(n-2x-1)-(n-x-1)+\frac{1}{2}\right)\,\sqrt{n+\frac{1}{n-2}}=0.
\end{align*}
Thus, $g(x)$ is a strictly increasing function. Since $1\leq k\leq k'\leq n-2$, by Theorem \ref{k0}, we obtain
\begin{align*}
ASO(G)&\leq k\,\sqrt{n+\frac{2}{n-2}}+{k\choose 2}\,\sqrt{n+\frac{1}{n-2}}\\[3mm]
&\quad+k\,(n-k-1)\,\sqrt{n-\frac{1}{2}+\frac{5}{2(2n-5)}}\\[3mm]
&\quad+{n-k-1\choose 2}\,\sqrt{n-1+\frac{1}{n-3}}\\[3mm]
&\leq k'\,\sqrt{n+\frac{2}{n-2}}+{k'\choose 2}\,\sqrt{n+\frac{1}{n-2}}\\[3mm]
&\quad+k'\,(n-k'-1)\,\sqrt{n-\frac{1}{2}+\frac{5}{2(2n-5)}}\nonumber\\[3mm]
&\quad+{n-k'-1\choose 2}\,\sqrt{n-1+\frac{1}{n-3}}.
\end{align*}
Moreover, both equalities in the above chain of inequalities hold if and only if $G$ is isomorphic to $(K_1\cup K_{n-k'-1})\vee K_{k'}$.
\end{proof}

\section{Concluding Remarks}\label{sec4}

The augmented Sombor (ASO) index is a recently introduced topological index that has already attracted attention, particularly for its chemical applications. In this work, we have investigated several mathematical properties of the ASO index. In particular, we establish the best possible upper bound for the ASO index of unicyclic graphs in terms of their order and characterized the corresponding extremal graphs. Also, we have shown that the ASO index decreases (increases, respectively) whenever an edge is removed (added, respectively) to a graph under a mild condition. Using this result, we have characterized the graphs maximizing the ASO index among all connected graphs of order $n$ with prescribed vertex connectivity (or edge connectivity) for $n\ge4$. Despite these contributions, the extremal behavior of the ASO index remains largely unexplored for many classical families of graphs, thereby presenting a variety of promising directions for further investigation.


\singlespacing


\begin{thebibliography}{00}
\bibitem{Ali-IJACM-17} A. Ali, A. A. Bhatti, Z. Raza, Further inequalities between vertex-degree-based topological indices, {\it Int. J. Appl. Comput. Math.} {\bf3} (2017) 1921--1930.


\bibitem{Ali-CJC-16} A. Ali, Z. Raza, A. A. Bhatti, Extremal pentagonal chains with respect to bond incident degree indices, {\it Canad. J. Chem.} {\bf94} (2016) 870--876.

\bibitem{Bondy-book} J. A. Bondy, U. S. R. Murty, {\it Graph Theory}, Springer, Heidelberg, 2008.

\bibitem{borovicanin17zagreb}B. Borovi\v{c}anin, K. C. Das, B. Furtula, I. Gutman, Bounds for Zagreb indices, {\it MATCH Commun. Math. Comput. Chem.\/} {\bf 78} (2017) 17--100.



\bibitem{Chartrand-16} G. Chartrand, L. Lesniak, P. Zhang, {\it Graphs} \& {\it Digraphs}, CRC Press, 2016.


\bibitem{chen24}M. Chen, Y. Zhu, Extremal unicyclic graphs of Sombor index, {\it Appl. Math. Comput.\/} {\bf 463} (2024) 128374.

\bibitem{cruz21chemical}R. Cruz, I. Gutman, J. Rada, Sombor index of chemical graphs, {\it Appl. Math. Comput.\/} {\bf 399} (2021) 126018.

\bibitem{cruz21unibicy}R. Cruz, J. Rada, Extremal values of the Sombor index in unicyclic and bicyclic graphs, {\it J. Math. Chem.\/} {\bf 59} (2021) 1098--1116.

\bibitem{cruz21trees}R. Cruz, J. Rada, J. M. Sigarreta, Sombor index of trees with at most three branch vertices, {\it Appl. Math. Comput.\/} {\bf 409} (2021) 126414.

\bibitem{das24open}
K. C. Das, Open problems on Sombor index of unicyclic and bicyclic graphs, {\it Appl. Math. Comput.\/} {\bf 473} (2024) 128644.

\bibitem{KD1}  K. C. Das, On the vertex degree function of graphs, {\it Comput. Appl. Math.\/} {\bf 44} (2025) 183.

\bibitem{das21symmetry}K. C. Das, A. S. \c{C}evik, I. N. Cangul, Y. Shang, On Sombor index, {\it Symmetry\/} {\bf 13} (2021) 140.

\bibitem{das22trees}K. C. Das, I. Gutman, On Sombor index of trees, {\it Appl. Math. Comput.\/} {\bf 412} (2022) 126575.

\bibitem{DGA1} K. C. Das, I. Gutman, A. Ali, Augmented Sombor Index, {\it MATCH Commun. Math. Comput. Chem.\/} {\bf 95} (2026) 523--547.

\bibitem{das21extremal}K. C. Das, Y. L. Shang, Some extremal graphs with respect to Sombor index, {\it Mathematics\/} {\bf 9} (2021) 1202.



\bibitem{deng21moleculartrees}H. Y. Deng, Z. K. Tang, R. F. Wu, Molecular trees with extremal values of Sombor indices, {\it Int. J. Quantum Chem.\/} {\bf 121} (2021) e26622.

\bibitem{Desmecht-JCIM-24} D. Desmecht, V. Dubois, Correlation of the molecular cross-sectional area of organic monofunctional compounds with topological descriptors, {\it J. Chem. Inf. Model.} {\bf64} (2024) 3248--3259.



\bibitem{gutman13degree}I. Gutman, Degree-based topological indices, {\it Croat. Chem. Acta\/} {\bf 86} (2013) 351--361.


\bibitem{gutman21geo}I. Gutman, Geometric approach to degree-based topological indices: Sombor indices, {\it MATCH Commun. Math. Comput. Chem.\/} {\bf 86} (2021) 11--16.

\bibitem{gutman24eu}I. Gutman, Relating Sombor and Euler indices, {\it Military Tech. Courier\/} {\bf 71} (2024) 1--12.



\bibitem{horoldagvaa21sombor}
B. Horoldagvaa, C. Xu, On Sombor index of graphs, {\it MATCH Commun. Math. Comput. Chem.\/} {\bf 86} (2021) 703--713.

\bibitem{Leite-24} L. S. G. Leite, S. Banerjee, Y. Wei, J. Elowitt, A. E. Clark, Modern chemical graph theory, {\it WIREs Comput. Mol. Sci.} {\bf14} (2024) \#e1729.


\bibitem{CAM3}H. Liu, Extremal problems on Sombor indices of unicyclic graphs with a given diameter, {\it Comput. Appl. Math.\/} {\bf 41} (2022) 138.

\bibitem{liu22review}H. Liu, I. Gutman, L. You, Y. Huang, Sombor index: review of extremal results and bounds, {\it J. Math. Chem.\/} {\bf 60} (2022) 771--798.

\bibitem{liu22tetracyclic}H. Liu, L. You, Y. Huang, Extremal Sombor indices of tetracyclic (chemical) graphs, {\it MATCH Commun. Math. Comput. Chem.\/} {\bf 88} (2022) 573--581.

\bibitem{CAM4}V. Maitreyi, S. Elumalai, S. Balachandran, H. Liu, The minimum Sombor index of trees with given number of pendant vertices, {\it Comp. Appl. Math.\/} {\bf 42} (2023) 331.

\bibitem{Milovanovic-21} I. Milovanovi\'c, E. Milovanovi\'c, A. Ali, M. Mateji\'c, Some results on
the Sombor indices of graphs, {\it Contrib. Math.} {\bf3} (2021) 59--67.

\bibitem{nithyaa24}P. Nithyaa, S. Elumalai, S. Balachandran, M. Masrec, Ordering unicyclic graphs with a fixed girth by Sombor indices, {\it MATCH Commun. Math. Comput. Chem.\/} {\bf 92} (2024) 205--224.


\bibitem{SA} W. A. Stein, {\it Sage Mathematics Software\/} (Version 6.8), The Sage Development Team, \url{http://www.sagemath.org, 2015}.

\bibitem{Todeschini-20-book} R. Todeschini, V. Consonni, {\it Handbook of Molecular Descriptors}, Wiley-VCH, Weinheim, 2000.

\bibitem{Trina-book} N. Trinajsti\'c, {\it Chemical Graph Theory},  CRC Press, Boca Raton, 1992.

\bibitem{Wagner-18} S. Wagner, H. Wang, {\it Introduction to Chemical Graph Theory},
        CRC Press, Boca Raton, 2018.
\end{thebibliography}
\end{document}